\documentclass[12pt,a4paper]{amsart}
\usepackage{tgtermes}

\calclayout

\usepackage{amsmath,latexsym,amssymb,amsthm,amscd,comment,enumerate,arydshln,mathtools,tikz}
\usepackage[all]{xy}
\usepackage[british]{babel}
\usepackage{url}
\usepackage{enumitem}

\usepackage{listings}
\usepackage{newfloat}
\DeclareFloatingEnvironment[
fileext=los,
listname={List of Algorithms},
name=Algorithm,
placement=tbhp,
within=section,
]{algorithm}

\usepackage[colorlinks=true,allcolors=blue]{hyperref}
\usepackage[capitalize,nameinlink,noabbrev]{cleveref}

\DeclareMathOperator{\codim}{codim}

\DeclareMathOperator{\stab}{Stab}

\newcommand{\C}{{\mathcal{D}}}

\newcommand{\QQ}{\mathbb{Q}}

\newcommand{\PP}{\mathbb{P}}

\renewcommand{\epsilon}{\varepsilon}

\newtheorem{thm}[equation]{Theorem}
\newtheorem{thmintro}{Theorem}

\crefname{thmintro}{Theorem}{Theorems}

\newtheorem{lem}[equation]{Lemma}
\newtheorem{cor}[equation]{Corollary}
\newtheorem{example}[equation]{Example}

\theoremstyle{remark}
\newtheorem{remark}[equation]{Remark}

\theoremstyle{definition}

\title[Irreducibility criteria for the preimages of a transverse variety...]{Irreducibility criteria for the preimages of a transverse variety under endomorphisms of products of elliptic curves}

\author{Riccardo Pengo}
\address{
	Riccardo Pengo - Institut für Algebra, Zahlentheorie und Diskrete Mathematik, Fakultät für Mathematik und Physik, Leibniz Universität Hannover, Welfengarten 1, 30167 Hannover, Germany}
\email{\href{mailto:pengo@math.uni-hannover.de}{\textcolor{black}{pengo@math.uni-hannover.de}}}

\author{Evelina Viada}
\address{Evelina Viada - Mathematisches Institut, Georg-August-Universität Göttingen, Bunsenstra{\ss}e 3-5, 37073 Göttingen, Germany}
\email{\href{mailto:evelina.viada@mathematik.uni-goettingen.de}{\textcolor{black}{evelina.viada@mathematik.uni-goettingen.de}}}
\date{\today}



\begin{document}

\begin{abstract}
	We provide two different proofs of an irreducibility criterion for the preimages of a transverse subvariety of a product of elliptic curves under a diagonal endomorphism of sufficiently large degree. 
	For curves, we present an arithmetic proof of the aforementioned irreducibility result, which enlightens connections to methods used in the context of the Torsion Anomalous Conjecture. On the other hand, we generalize the result for higher dimensional varieties using a more  geometric approach.
	Finally, we give some applications of these results. More precisely, we establish the irreducibility of some explicit families of polynomials, we provide new estimates for the normalized heights of certain intersections and images, and we give new lower bounds for the essential minima of images.
\end{abstract}
\maketitle

\section{Introduction}

The aim of the present note is to prove by two different methods a criterion for the irreducibility of the preimage of a transverse subvariety of a product of elliptic curves under a diagonal endomorphism. 
It is well known that to establish the irreducibility of a variety is a difficult matter. Here, we show that the assumption of transversality of the variety is a necessary and sufficient condition to ensure the irreducibility of its preimage by diagonal endomorphisms of large degree. 
In particular, we make this largeness explicit with a precise bound depending only on the partial degrees of the variety.  
Such a question, which has a geometric flavour, can in fact be approached by arithmetic means, at least in the case of curves. 
A key role is played by the precise formulas for the degrees of the preimages of subvarieties, proven by Hindry in \cite[Lemme~6]{grado} (see also \cref{lem:degree_preimage}). 

\subsection*{Notation}
We fix an algebraic closure $\overline{\QQ}$ of the rational numbers $\QQ$.
In this article, a \emph{variety} is an algebraic variety defined over $\overline{\QQ}$. 

\subsubsection*{Elliptic curves}
We fix an integer $N \geq 1$ and $N$ elliptic curves $E_1,\dots,E_N$ defined over $\overline{\QQ}$. We write
\[A_N := E_1 \times \dots \times E_N\] 
for their product.
Moreover, for each $i \in \{1,\dots,N\}$ we fix a Weierstrass model
\begin{equation} \label{eq:Weierstrass}
	E_i \colon y^2 = x^3 + A_ix + B_i,
\end{equation}
where $A_i$ and $B_i$ are algebraic integers.
These Weierstrass models induce the embedding
\begin{equation} \label{eq:fixed_embedding}
	A_N \hookrightarrow (\PP^{2})^N \hookrightarrow \PP^{m}
\end{equation}
where $m = 3^N - 1$, which is given by the composition of the product of the natural inclusions $E_i \hookrightarrow \mathbb{P}^2$ determined by the Weierstrass equations \eqref{eq:Weierstrass}, together with the Segre embedding $(\mathbb{P}^2)^N \hookrightarrow \mathbb{P}^m$.

\subsubsection*{Endomorphisms}
We write $\mathrm{End}(A_N)$ for the endomorphism ring of $A_N$ considered as an abelian variety. We recall as well that $\phi \in \mathrm{End}(A_N)$ is an \emph{isogeny} if it is surjective and has finite kernel. In this case, the \emph{degree} of $\phi$ is defined as $\deg(\phi) := \lvert \ker(\phi) \rvert$.
Every endomorphism $\phi \in \mathrm{End}(A_N)$ can be represented by a matrix $\phi = (\phi_{i,j})_{i,j = 1}^N$, where each $\phi_{i,j} \colon E_i \to E_j$ is either an isogeny or the trivial map. 
In particular, every tuple $(\alpha_1,\dots,\alpha_N) \in \prod_i \mathrm{End}(E_i)$ induces a diagonal endomorphism
\begin{equation} \label{eq:diagonal_endomorphism}
	[\alpha_1,\dots,\alpha_N] := \begin{pmatrix}
		\alpha_1 & & \\
		& \ddots & \\
		& & \alpha_N
	\end{pmatrix}.
\end{equation}

\subsubsection*{Transversality}
A subvariety $V\subset A_N$  is a \emph{translate} (respectively a torsion variety) if it is a finite union of translates of proper algebraic subgroups of $A_N$ by points (respectively by torsion points). Moreover, an irreducible subvariety $V\subset A_N$  is \emph{transverse}, (respectively \emph{weak-transverse}), if it is not contained in any translate, (respectively in any torsion variety).

\subsubsection*{Degrees} Given a sub-variety $V \subseteq \mathbb{P}^{n}$, we let $\deg(V)$ denote its degree, defined as the number of points of the intersection between $V$ and $\dim(V)$ generic hyperplanes. More generally, if $V$ is a sub-variety of a multiprojective space $\mathbb{P}^{n_1} \times \dots \times \mathbb{P}^{n_r}$, and $I = (i_1,\dots,i_r) \in \mathbb{N}^r$ is a tuple of non-negative integers such that $i_1 + \dots + i_r = \dim(V)$, one can define a multiprojective degree $\deg_I(V)$ by counting the number of points of the intersection of $V$ with $\dim(V)$ generic hyperplanes $H_1,\dots,H_{\dim(V)}$ such that for every $j \in \{1,\dots,r\}$ and every $k \in \{i_1 + \dots + i_{j-1} + 1,i_1 + \dots + i_j\}$ we have that \[H_k = \mathbb{P}^{n_1} \times \cdots \times \mathbb{P}^{n_{j-1}} \times H_k' \times \mathbb{P}^{n_{j+1}} \times \dots \times \mathbb{P}^{n_r},\] where $H_k' \subseteq \mathbb{P}^{n_j}$ is a generic hyperplane. For more details, we refer the interested reader to \cref{sec:degrees}.

\subsection*{Main results}	
	We are now ready to state the main results of this note.
	The following theorem provides a first irreducibility criterion for the preimages of transverse curves $C \subseteq A_N$.
	
\begin{thmintro} \label{thm:curves_intro}
	Let  $p_1,\dots,p_N \in \mathbb{Z}$  be  prime numbers defining the  endomorphism $[p_1,\dots,p_N]$ of $ A_N $. Then, for every transverse curve $C \subseteq A_N $  such that 
	\[
	\lvert p_j \rvert \geq \deg(C) N 3^{N - 1}
	\] 
	for each $j \in \{1,\dots,N\}$, the preimage $$[p_1,\dots,p_N]^{-1}(C)$$ is transverse (and, in particular, irreducible).
\end{thmintro}

\begin{remark}
	The transversality hypothesis required in \cref{thm:curves_intro} cannot be dropped, as the following counterexample easily shows.
	Suppose that $N = N_1 + N_2$ and let $A_{N_1} = E_1 \times \dots \times E_{N_1}$ and $A_{N_2} = E_{N_1 + 1} \times \dots \times E_{N}$. Moreover, fix a curve $C \subseteq A_{N_1}$, and let $O \in A_{N_2}(\overline{\QQ})$ be the origin.
	Then, for every $n \in \mathbb{Z}$, the preimage of $C \times \{O\}$ by the map $[n,\dots,n]$ is given by $\C\times A_{N_2}[n]$, where $\C$ is the preimage of $C$ under the multiplication by $n$ on $A_{N_1}$. Even if $\C$ will be irreducible when $\lvert n \rvert$ is large enough, the product $\C \times A_{N_2}[n]$ is clearly reducible, because its irreducible components are given by $\C\times \{T\}$, where $T \in A_{N_2}[n]$ runs over the $n$-th torsion points of $A_{N_2}$.
	Note moreover that transversality cannot be replaced by weak-transversality in the statement of \cref{thm:curves_intro}. To see this, it suffices to consider the weak-transverse curve $C \times \{P\}$, where $P \in A_{N_2}(\overline{\mathbb{Q}})$ is any non-torsion point.
\end{remark}

The previous theorem will be proven in \cref{sec:curves}. Our proof is arithmetic in nature, and relies on the arithmetic Bézout theorem of Philippon \cite{patrice} and the comparison between the height and the essential minimum of a variety, proven by Zhang \cite{Zha_Small}. 
Moreover, this proof was inspired by some of the techniques used in the proof of the Torsion Anomalous Conjecture for curves inside a product of elliptic curves (see \cite{Viada_2008} and \cite{CVV_Forum}).

The aforementioned conjecture, proposed by Bombieri, Masser and Zannier \cite{BMZ1}, is still open  for general varieties, despite several partial results. We refer to  \cite{libroZannier} or \cite{Viada_2016} for an introduction to the question.
(Weak)-transversality is the central geometric assumption in this conjecture.
Several works, such as \cite{Bombieri_Masser_Zannier_1999,Maurin_2008,Habegger_2009}, as well as the works cited in \cite[Section~3]{Viada_2016}, have pointed out the connection between (weak)-transversality and the arithmetic of a variety. 
Other results, like \cite{Viada_2021}, shed some light on how difficult and important it is to determine whether a curve is or not transverse. 
This inspired our \cref{thm:curves_intro}.

On the other hand, our arithmetic proof of \cref{thm:curves_intro} does not generalize easily to higher dimensional varieties, as we point out in \cref{rmk:arithmetic_not_generalizing}.
This prompted us to find a different, geometric proof of \cref{thm:curves_intro}, which generalizes to higher dimensional varieties and yields slightly different bounds. 
The best version of the bounds that we obtain is portrayed in \cref{thm:stronger_general_varieties}, and involves the notion of multiprojective degree, which we recall in \cref{sec:degrees}.
To ease notations, we prefer to state here a weaker version of our main result, which is a corollary of \cref{thm:stronger_general_varieties}, as we show in \cref{lem:varieties_strong_to_weak}.

\begin{thmintro} \label{thm:varieties_intro}
		Let $V \subseteq A_N$ be a transverse variety. 
		Moreover, let 
		\[
		(\alpha_1,\dots,\alpha_N) \in \prod_{j=1}^N \mathrm{End}(E_j)
		\] 
		be a tuple of isogenies such that for every prime $p \mid \deg(\alpha_1) \cdots \deg(\alpha_N)$ we have that 
		\[
			p > \dim(V)! \deg(V).
		\] 
		Then, the preimage $$ [\alpha_1,\dots,\alpha_N]^{-1}(V)$$ is transverse (and, in particular, irreducible).
\end{thmintro}

This theorem can be stated in a more direct way when $[\alpha_1,\dots,\alpha_N]$ is an integral multiple of the identity, as we show in \cref{cor:multiple_identity}.
In \cref{sec:examples}, we give some applications of our main results. More precisely, we show that the equations defining the preimages considered in \cref{thm:curves_intro,thm:varieties_intro} can quickly become very complicated.
In particular, the irreducibility guaranteed by our results is difficult to check by other means. Thus, using our theorems, we provide new families of irreducible polynomials.
Finally, we show how our theorems can be used to improve the upper bound on the height of certain intersections of $A_N$ obtained  using the arithmetic Bézout theorem \eqref{ineq:bezout}, and the lower bound on the essential minimum of a subvariety provided by a theorem of Galateau \cite{Galateau_2010}, portrayed in \eqref{ineq:Galateau}. 

\subsection*{Acknowledgments} We thank the anonymous referee for their precious comments and suggestions.

\subsection*{Funding}

The first author is grateful to the Max Planck Institute for Mathematics in Bonn for providing excellent working conditions, great hospitality and financial support.
Moreover, the first author thanks the research projects ``Motivic homotopy, quadratic invariants and diagonal classes'' (ANR-21-CE40-0015) and IRN GANDA for their financial support, and the University of Göttingen for its hospitality.

\vspace{\baselineskip}
\noindent
\framebox[\textwidth]{
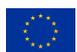
	\begin{tabular*}{0.96\textwidth}{@{\extracolsep{\fill} }cp{0.84\textwidth}}
		% The EU emblem
		\raisebox{-0.7\height}{%
			\begin{tikzpicture}[y=0.80pt, x=0.8pt, yscale=-1, inner sep=0pt, outer sep=0pt, 
				scale=0.12]
				\definecolor{c003399}{RGB}{0,51,153}
				\definecolor{cffcc00}{RGB}{255,204,0}
				\begin{scope}[shift={(0,-872.36218)}]
					\path[shift={(0,872.36218)},fill=c003399,nonzero rule] (0.0000,0.0000) rectangle (270.0000,180.0000);
					\foreach \myshift in 
					{(0,812.36218), (0,932.36218), 
						(60.0,872.36218), (-60.0,872.36218), 
						(30.0,820.36218), (-30.0,820.36218),
						(30.0,924.36218), (-30.0,924.36218),
						(-52.0,842.36218), (52.0,842.36218), 
						(52.0,902.36218), (-52.0,902.36218)}
					\path[shift=\myshift,fill=cffcc00,nonzero rule] (135.0000,80.0000) -- (137.2453,86.9096) -- (144.5106,86.9098) -- (138.6330,91.1804) -- (140.8778,98.0902) -- (135.0000,93.8200) -- (129.1222,98.0902) -- (131.3670,91.1804) -- (125.4894,86.9098) -- (132.7547,86.9096) -- cycle;
				\end{scope}
				%\draw[very thin,dashed] (current bounding box.south west) rectangle               (current bounding box.north east);
			\end{tikzpicture}%
		}
		&
		Riccardo Pengo received funding from the European Research Council (ERC) under the European Union’s Horizon 2020 research and innovation programme (grant agreement number 945714).
	\end{tabular*}
}

\section{Some geometric preliminaries}
\label{sec:preliminaries}

The aim of this section is to prove some basic lemmas which will be used in the proofs of both of our main results.

\subsection{Subvarieties and transversality}

First of all, we observe that any translation of a transverse subvariety of $A_N$ by a point $P \in A_N(\overline{\QQ})$ is again transverse. Moreover, the following lemma links transversality to irreducibility.

\begin{lem}\label{trairr}
	Fix $N, M \in \mathbb{Z}_{\geq 1}$. 
	Then, for any surjective morphism of abelian varieties $\phi:A_N \twoheadrightarrow A_M$, the preimage
	$\phi^{-1}(V)$ of a transverse variety $V \subseteq A_M$ is transverse if and only if it is irreducible.
\end{lem}
\begin{proof}
	Transversality implies irreducibility by definition. 
	On the other hand, suppose that $\phi^{-1}(V)$ is irreducible, and assume by contradiction that there exists a translate $B + P \subsetneq A$ which contains $\phi^{-1}(V)$, where $B \subsetneq A_N$ is a proper algebraic subgroup, and $P \in A_N(\overline{\QQ})$.
	Then $B \supseteq \ker(\phi)$, and so $\phi(B) \neq A_M$ thanks to the subgroup correspondence theorem.
	Moreover, we have that 
	\[
	V = \phi(\phi^{-1}(V)) \subseteq \phi(B + P) = \phi(B)+\phi(P),
	\] 
	because $\phi$ is surjective by assumption. 
	These considerations imply that $V$ is not transverse, because $\phi(B)+\phi(P)$ is a translate containing $V$. However, this fact contradicts our assumptions, and implies therefore that $\phi^{-1}(V)$ must be transverse, as we wanted to show.
\end{proof}

This shows that, to study the transversality of preimages of transverse varieties it is sufficient to study their irreducibility.

\subsection{Action of the Kernel on the fibres}
The following lemma  describes some invariants of  the irreducible components of the preimage  of a transverse subvariety under an isogeny.
\begin{lem} \label{degcomponenti}
	Let $V \subseteq A_N$ be a transverse subvariety, and let $\phi$ be an isogeny of $A_N$.
	Then, the action of the subgroup $\ker(\phi)$ on the irreducible components of $\phi^{-1}(V)$ is transitive. 
	In particular, each irreducible component of $\phi^{-1}(V)$ surjects on $V$, and all these irreducible components have the same dimension, degree and stabilizer.
\end{lem}
\begin{proof}
	First of all, observe that there exists an irreducible component $W \subseteq \phi^{-1}(V)$ of maximal dimension such that $\phi(W) = V$.
	Indeed, if this was not the case, $\phi(\phi^{-1}(V)) = V$ would be a closed subvariety of $V$ which has strictly smaller dimension, which is clearly absurd.
	
	Now, let $W' \subseteq \phi^{-1}(V)$ be any other irreducible component, and pick any point $P \in W(\overline{\QQ})$. Then, there exists a point $P' \in W'(\overline{\QQ})$ such that $\phi(P) = \phi(P')$, which implies that $P = P' + T$ for some point $T \in \ker(\phi)$.
	Therefore, for every $P \in W(\overline{\mathbb{Q}})$ there exists $T \in \ker(\phi)$ such that $P \in W' + T$, or in other words $W \subseteq \bigcup_{T \in \ker(\phi)} W' + T$. Since both $W$ and $W'$ are irreducible components, and $W$ has maximal dimension, this implies that $W = W' + T$ for some $T \in \ker(\phi)$. Therefore, $\phi(W') = \phi(W) = V$, and the action of $\ker(\phi)$ on the irreducible components of $\phi^{-1}(V)$ is transitive.
	This shows in particular that all the components have the same dimension and degree, because these two quantities are invariant under translations.
	
	To conclude, show that any variety $W \subseteq A_N$ and any translate $W+T$ by a point $T \in A_N(\overline{\QQ})$ have the same stabilizer. 
	Indeed, if $Q \in \stab(W)$ then $Q + (W+T) = (Q + W) + T = W + T = W'$, which implies that $\stab(W) \subseteq \stab(W')$. Similarly, any $Q' \in \stab(W + T)$ stabilizes $W = (W+T) - T$, which allows us to conclude that $\stab(W) = \stab(W')$.
\end{proof}

\subsection{Composing multiplication maps}
The aim of this subsection is to prove a result which relates the irreducibility of the preimages of a transverse variety $V \subseteq A_N$ by different multiplication maps.
To do so, we  need to single out a specific class of factorizations of the diagonal endomorphisms $[\alpha_1,\dots,\alpha_N]$ introduced in \eqref{eq:diagonal_endomorphism}.
First of all, let us observe that each of these endomorphisms can be factored as
\begin{equation} \label{eq:canonical_factorization}
	[\alpha_1,\dots,\alpha_n] = \begin{pmatrix}
		\alpha_1 & & & \\
		& 1 & & \\
		& & \ddots & \\
		& & & 1
	\end{pmatrix} \cdot 
	\begin{pmatrix}
		1 & & & \\
		& \alpha_2 & & \\
		& & \ddots & \\
		& & & 1
	\end{pmatrix}
	\cdots \begin{pmatrix}
		1 & & & \\
		& 1 & & \\
		& & \ddots & \\
		& & & \alpha_N
	\end{pmatrix}.
\end{equation}
In other words, we have that $[\alpha_1,\dots,\alpha_N] = [\alpha_1]_1 \circ \dots \circ [\alpha_N]_N$, where we define
\[
	[\alpha]_j := \begin{pmatrix}
		1 & & & & \\
		& \ddots & & & \\
		& & \alpha & & \\
		& & & \ddots & \\
		& & & & 1
	\end{pmatrix}
\]
for every $j \in \{1,\dots,N\}$ and $\alpha \in \mathrm{End}(E_j)$.

More generally, we say that a factor of $[\alpha_1,\dots,\alpha_N]$ is \emph{admissible} if it is the composition of some of the factors $[\alpha_1]_1,\dots,[\alpha_N]_N$.
For example, writing
\[
	[\alpha_1,\dots,\alpha_5] = 
	\begin{pmatrix}
		1 & & & & \\
		& \alpha_2 & & & \\
		& & 1 & & \\
		& & & \alpha_4 & \\
		& & & & 1
	\end{pmatrix} \cdot 
	\begin{pmatrix}
		\alpha_1 & & & & \\
		& 1 & & & \\
		& & 1 & & \\
		& & & 1 & \\
		& & & & \alpha_5
	\end{pmatrix}
	\cdot
	\begin{pmatrix}
		1 & & & & \\
		& 1 & & & \\
		& & \alpha_3 & & \\
		& & & 1 & \\
		& & & & 1
	\end{pmatrix}
\]
yields a factorization of $[\alpha_1,\dots,\alpha_5]$ into admissible factors.

With this definition at hand, we show in the following lemma that the transversality of the preimage of a transverse variety by a diagonal endomorphism is equivalent to the transversality of each of its preimages under the factors of any admissible factorization. 

\begin{lem} \label{lem:factorization}
	Let $\phi = [\alpha_1,\dots,\alpha_N]$ be a diagonal endomorphism of $A_N$.
	Fix moreover a transverse variety $V \subseteq A_N$, and a factorization $\phi = f_1 \circ \dots \circ f_r$ into admissible factors.
	Then, 
	\begin{itemize}
	\item $\phi^{-1}(V)$ is transverse if and only if $f_j^{-1}(V)$ is transverse for every $j \in \{1,\dots,r\}$. \end{itemize}
		
	Equivalently,\begin{itemize} \item $\phi^{-1}(V)$ is irreducible if and only if $f_j^{-1}(V)$ is irreducible for every $j \in \{1,\dots,r\}$.
	\end{itemize}
\end{lem}
\begin{proof}  
	The equivalence of the two statements is due to \cref{trairr}. We now prove the transversality statement.
	
	First of all, let us assume that $\phi^{-1}(V)$ is transverse, which implies by definition that $\phi^{-1}(V)$ is irreducible.
	We also observe that all the endomorphisms $f_1,\dots,f_r$ commute, because they are represented by diagonal matrices.
	Therefore, for every $j \in \{1,\dots,r\}$ the variety 
	\[
	f_j^{-1}(V) = (f_1 \circ \dots \circ f_{j-1} \circ f_{j+1} \circ \dots \circ f_r)(\phi^{-1}(V))
	\]
	will also be irreducible, since irreducibility is preserved under images.
	This allows us to conclude that $f_j^{-1}(V)$ is transverse, thanks to \cref{trairr}.
	
	Conversely, suppose that $f_j^{-1}(V)$ is transverse for every $j \in \{1,\dots,r\}$. 
	In order to prove that $\phi^{-1}(V)$ is transverse, we will prove its irreducibility, which suffices thanks to \cref{trairr}. To this aim, we proceed by induction on $r$. 
	
	To deal with the base case, let us suppose that $r = 2$. Then, for any given irreducible component $W \subseteq \phi^{-1}(V)$ we have that $f_1^{-1}(V) = f_2(W)$, thanks to \cref{degcomponenti}. 
	This implies that $\stab(W) \supseteq \ker(f_1)$, because $\stab(f_2(W)) = \stab(f_1^{-1}(V)) \supseteq \ker(f_1)$, and $f_2$ acts only on components which are left unchanged by the action of $\ker(f_1)$, since the factorization $f = f_1 \circ f_2$ is admissible. Analogously, we see that $\stab(W) \supseteq \ker(f_2)$ because $f_1(W) = f_2^{-1}(V)$.
	Therefore, we have that
	\[
	\stab(W) \supseteq \stab(f_1^{-1}(V)) + \stab(f_2^{-1}(V)),
	\] 
	which also implies that $\stab(W) \supseteq \ker(f_1) + \ker(f_2) = \ker(\phi)$, because the preimage of any variety by an endomorphism is stabilized by its kernel.
	This implies the irreducibility of $\phi^{-1}(V)$, as follows.
	We observe that $\phi(W) = V$ because $V$ is irreducible, which implies that $W = \phi^{-1}(V)$, because $W$ is stabilized by $\ker(\phi)$.
	
	To conclude this proof, let us proceed with the inductive step of our proof. If $r \geq 3,$ we know by the basis of the induction that $g^{-1}(V)$ is irreducible, where $g := f_{N-1} \circ f_N$. Therefore, we see by the inductive hypothesis that $\phi^{-1}(V) = (f_1 \circ \dots \circ f_{N-2} \circ g)^{-1}(V)$ is irreducible, as we wanted to show. 
\end{proof}

\begin{remark}
	In particular, the previous lemma shows that if $\phi = [\alpha_1,\dots,\alpha_n]$ and $V \subseteq A_N$ is transverse, the irreducibility of $[\alpha_j]_j^{-1}(V)$ for every $j \in \{1,\dots,N\}$ suffices to guarantee that $V$ is itself irreducible and transverse.
\end{remark}

\subsection{Degrees of preimages}
\label{sec:degrees}

In this subsection, we first  recall a central  result of Hindry (see \cite[Lemma~6]{grado}) concerning the geometric degree of preimages of algebraic varieties under group homomorphisms. Meanwhile, we introduce some relevant notation and we prove some preliminary lemmas on degrees.

In order to do so, let us recall that for every zero cycle $\xi := \sum_{j=1}^k a_j P_j$ on an algebraic variety $X$, where $a_1,\dots,a_k \in \mathbb{Z}$ and $P_1,\dots,P_k \in X(\overline{\QQ})$, we define $\deg_0(\xi) := \sum_{j = 1}^m a_k$.
Now, fix an embedding $\iota \colon X \hookrightarrow \mathbb{P}^n$, given by a very ample divisor $H \in \mathrm{Div}(X)$. Then, we have an associated notion of degree of a closed subvariety $Y \subseteq X$, which is given by $\deg(Y) := \deg_0(Y \cdot H^{\dim(V)})$.
Moreover, suppose that $X = X_1 \times \dots \times X_r$, and that the embedding $\iota$ factors as 
\[
\iota \colon X = X_1 \times \dots \times X_r \hookrightarrow \mathbb{P}^{n_1} \times \dots \times \mathbb{P}^{n_r} \hookrightarrow \mathbb{P}^n
\] 
where the first embedding is the product of some embeddings $\iota_j \colon X_j \hookrightarrow \mathbb{P}^{n_j}$, for $j \in \{1,\dots,r\}$, which correspond to some very ample divisors $H_j' \in \mathrm{Div}(X_j)$, and the second embedding 
\[
\mathbb{P}^{n_1} \times \dots \times \mathbb{P}^{n_r} \hookrightarrow \mathbb{P}^n
\]
is the Segre embedding, so that $n = (n_1+1) \cdots (n_r + 1) - 1$.
Then, the first embedding allows one to define a new notion of multiprojective degree. 
More precisely, for every tuple $I = (i_1,\dots,i_r) \in \mathbb{N}^r$ and every subvariety $Y \subseteq X$ such that $i_1+\dots+i_r = \dim(Y)$, one defines the multiprojective degree
\begin{equation} \label{eq:multiprojective_degree}
	\deg_I(Y) := \deg_0(Y \cdot H_1^{i_1} \cdots H_n^{i_n})
\end{equation}
where $H_j := H_j' \times \prod_{i \neq j} X_j$ for every $j \in \{1,\dots,r\}$. 
Moreover, we have that $H = H_1 + \dots + H_n$, as follows from the basic properties of the Segre embedding. Therefore, we see from the multinomial theorem that the degree of any closed subvariety $Y \subseteq X$ can be expressed as
\begin{equation} \label{eq:Hilbert_Samuel}
	\deg(Y) = \dim(Y)! \sum_{I} \frac{\deg_I(Y)}{I!},
\end{equation}
where the sum runs over all the $r$-tuples $I = (i_1,\dots,i_r) \in \mathbb{N}^r$ such that $i_1 + \dots + i_r = \dim(Y)$, and $I! := i_1! \cdots i_r!$.
These and further properties of multiprojective degrees can be found in \cite[\S~3]{patrice_zeros}.

Now, let us point out that these multiprojective degrees are useful if one wants to express the degree of the preimage of a subvariety $V \subseteq A_N$ in terms of the degree of the subvariety itself, as expressed by the following lemma.

\begin{lem} \label{lem:degree_preimage}
	Let $V \subseteq A_N$ be a subvariety, and $(\alpha_1,\dots,\alpha_N) \in \prod_{j=1}^N \mathrm{End}(E_j)$. Then, we have that
	\begin{equation} \label{eq:degree_preimage}
		\deg([\alpha_1,\dots,\alpha_N]^{-1}(V)) = \dim(V)! \sum_{J} \left( \prod_{k \not\in J} \deg(\alpha_k) \right) \deg_{I_J}(V) 
	\end{equation}
	where the sum runs over the subsets $J \subseteq \{1,\dots,N\}$ with $\lvert J \rvert = \dim(V)$, and 
	\[
	I_J = (i_{J,1},\dots,i_{J,N}) \in \mathbb{N}^N
	\] is the tuple defined by setting $i_{J,j} = 1$ if $j \in J$, and $i_{J,j} = 0$ otherwise.
\end{lem}
\begin{proof}
	First of all, let us observe that for every $j \in \{1,\dots,N\}$ the divisor $H_j' \in \mathrm{Div}(E_j)$ corresponding to the embedding $E_j \hookrightarrow \mathbb{P}^2$ is simply given by $H_j' = 3 (0 \colon 1 \colon 0)$.
	Therefore, the divisor $H \in \mathrm{Div}(A_N)$ corresponding to the embedding \eqref{eq:fixed_embedding} is given by $H = H_1 + \dots + H_N$, where 
	\begin{equation} \label{eq:embedding_divisors}
		H_j := 3 \cdot E_1 \times \dots \times E_{j-1} \times \{(0 \colon 1 \colon 0)\} \times E_{j+1} \times \dots \times E_N
	\end{equation}
	for every $j \in \{1,\dots,N\}$.
	In particular, we see that $H_j^2 = 0$ for every $j \in \{1,\dots,n\}$, because in the $j$-th factor $H_j^2$ reduces to the intersection of two generic points, which is empty. 
	Therefore, we see that $\deg_I(V) = 0$ for every $I = (i_1,\dots,i_N) \in \mathbb{N}^N$ such that there exists $j \in \{1,\dots,N\}$ with $i_j \geq 2$. 
	Since all the remaining tuples are of the form $I_J$ for some subset $J \subseteq \{1,\dots,N\}$, and $I_J! = 1$ for each of these tuples, we see from \eqref{eq:Hilbert_Samuel} that
	\[
	\deg([\alpha_1,\dots,\alpha_N]^{-1}(V)) = \dim(V)! \sum_J \deg_{I_J}([\alpha_1,\dots,\alpha_N]^{-1}(V))
	\]
	where the sum runs over all the subsets $J \subseteq \{1,\dots,N\}$ such that $\lvert J \rvert = \dim(V)$.
	To conclude our proof, it suffices to observe that \[\deg_{I_J}([\alpha_1,\dots,\alpha_N]^{-1}(V)) = \deg_{I_J}(V) \cdot \prod_{k \not\in J} \deg(\alpha_k) \] for every $J \subseteq \{1,\dots,N\}$ such that $\lvert J \rvert = \dim(V)$, as follows from \cite[Lemma~6]{grado}.
\end{proof}

To conclude this subsection, we point out that one can always translate a variety with finite stabilizers in such a way that the intersection between the resulting translated variety and any family of hyperplanes used to compute a multiprojective degree $\deg_I(V)$ is a finite set of points.

\begin{lem}
	Let $V$ be a transverse variety in $A_N$ with finite stabilizer. 
	Then, there exists a closed point $P \in A_N(\overline{\QQ})$ such that for every tuple $I = (i_1,\dots,i_N) \in \mathbb{N}^N$ with $i_1+\dots+i_N = \dim(V)$, the support of the intersection
	\[
	(V+P)  \cdot H_1^{i_1}\cdots H_N^{i_N}
	\] 
	is a finite set of points.
\end{lem}
\begin{proof}
	Suppose that there exists an intersection $V\cdot H_1^{i_1}\cdots H_N^{i_N}$ that is empty or of positive dimension. This implies that $i_j \in \{0,1\}$ for every $j \in \{1,\dots,N\}$, as we mentioned in the proof of \cref{lem:degree_preimage}. 
	Up to a reordering of the coordinates, we can assume that $I$ is given by the index 
	\[
	I=(1,\dots,1,0,\dots,0)
	\] 
	where the first $\dim(V)$ indices are $1$. 
	Then, let $H := H_1\cap \dots \cap H_{\dim V}$, and observe that the intersection $V\cap H$ is also empty or of positive dimension.
	Now, the image of $V$ under the canonical projection 
	\[
	\pi := \pi_I \colon A_N \twoheadrightarrow A_N/H=E_1\times \dots \times E_{\dim V}
	\] 
	is a closed subvariety of $A_N/H$, because $\pi$ is closed. 
	Moreover, $\dim(\pi(V)) = \dim(V)$.
	Indeed, if we had by contradiction that $\dim(\pi(V))<\dim V$, the fibre $\pi^{-1}(Q)$ of a generic point $Q \in \pi(V)$ would be stable under the action of a positive dimensional subgroup $S \subseteq H$. Therefore, we would have that $S \subset \stab(V)$, which contradicts the assumption that $\stab(V)$ is finite. 
	This implies that
	\[\dim(\pi(V)) = \dim(V)=\dim(A_N/H),\] 
	which shows that $ \pi_{|_V}$ is surjective, and that its generic fibre is zero dimensional. 
	Since the possible tuples $I$ are finitely many, there exists a point $P \in A_N(\overline{\QQ})$ such that for all $I$ the fibre $\pi_I^{-1}(0)\cap  (V+P)$  is generic and therefore zero dimensional. \end{proof}

This lemma, which could be potentially useful in future geometric applications, will not be applied in the rest of this paper.

\section{Preimages of transverse curves: an arithmetic approach}
\label{sec:curves}

In  this section we prove   \cref{thm:curves_intro}  using an arithmetic method. Among others, we use the arithmetic Bézout theorem \eqref{ineq:bezout} together with Zhang's inequality \eqref{ineq:zhang}.  Thus, we use arithmetic information in order to understand the transversality of a variety, which is a geometric notion.

\subsection{Some Diophantine inequalities}

The aim of this subsection is to recall the three fundamental inequalities \eqref{ineq:bezout}, \eqref{ineq:zhang} and \eqref{ineq:Zimmer}, and to recall some notation needed in the sequel.

\subsubsection*{Heights}
To start, we let $h_2 \colon \mathbb{P}^m(\overline{\QQ}) \to \mathbb{R}$ denote the \emph{Faltings (or Fubini-Study) height} of points, which is defined as
\[
	h_2(P) := \sum_{v \in M_K^0} \frac{[K_v \colon \mathbb{Q}_v]}{[K \colon \mathbb{Q}]} \log\max_{j=1,\dots,m}\{\lvert P_j \rvert_v\} + \sum_{v \in M_K^\infty} \frac{[K_v \colon \mathbb{Q}_v]}{2 [K \colon \mathbb{Q}]} \log\left( \sum_{j=1}^m \lvert P_j \rvert^2_v \right)
\]
where $K$ is any number field over which $P$ is defined, and $M_K^0$ (respectively $M_K^\infty$) denotes the set of finite (resp. infinite) places of $K$.
This height can be extended to subvarieties $V \subseteq \mathbb{P}^m$ in several ways, and in this paper we follow the convention introduced by Philippon in \cite[Section~2.B]{patriceI}.
Moreover, for every subvariety $V \subseteq A_N$, we let $\hat{h}(V)$ denote the \emph{Néron-Tate height} of $V$ associated to our fixed embedding \eqref{eq:fixed_embedding}, which is defined as in \cite[Page~281]{patriceI}.
We also let $h \colon \mathbb{P}^m(\overline{\QQ}) \to \mathbb{R}$ denote the \emph{logarithmic Weil height} of points, defined as
\[
	h(P) := \sum_{v \in M_K} \frac{[K_v \colon \mathbb{Q}_v]}{[K \colon \mathbb{Q}]} \log\max_{j=1,\dots,m}\{\lvert P_j \rvert_v\},
\]
where $K$ is any number field over which $P$ is defined, and $M_K$ denotes its set of places.
Finally, if $K$ is a number field we let $h_\infty \colon K \to \mathbb{R}$ denote the \emph{Archimedean contribution to the Weil height}, which is defined to be
\[
	h_\infty(x) := \sum_{v \in M_K^\infty} \frac{[K_v \colon \mathbb{Q}_v]}{[K \colon \mathbb{Q}]} \log\max(\lvert x \rvert_v,1).
\]
We note in particular that this function \emph{depends} on the number field $K$.

\subsubsection*{The arithmetic B\'ezout theorem}

Suppose now that $X,Y \subseteq \mathbb{P}^m$ are irreducible subvarieties, and let $Z_1,\dots,Z_g$ be the irreducible components of $X \cap Y$.
Then, the arithmetic Bézout theorem, which was proven by Philippon in \cite[Theorem~3]{patrice}, implies that
\begin{equation} \label{ineq:bezout}
	\sum_{i=1}^g h_2(Z_i)\leq\deg(X)h_2(Y)+\deg(Y)h_2(X)+c_0(\dim X,\dim Y, m)\deg(X)\deg(Y),
\end{equation}
where the function $c_0 \colon \mathbb{N}^3 \to \mathbb{Q}$ admits the explicit expression
\begin{equation}\label{costaBez}
	\begin{aligned}
		c_0(d_1,d_2,m) &= \left(\sum_{i=0}^{d_1}\sum_{j=0}^{d_2} \frac{1}{2(i+j+1)}\right)+\left(m-\frac{d_1+d_2}{2}\right)\log2 \\
		&= \frac{1}{2} \left( (d_1 + d_2 + 1) \mathcal{H}_{d_1 + d_2 + 2} - (d_1 + 1) \mathcal{H}_{d_1 + 1} - (d_2 + 1) \mathcal{H}_{d_2 + 1} \right) \\
		&+ \left(m-\frac{d_1+d_2}{2}\right)\log2,
	\end{aligned}
\end{equation}
which features the harmonic numbers $\mathcal{H}_k := \sum_{j=1}^k \frac{1}{j}$.

\subsubsection*{Zhang's inequality}

We will now recall another seminal inequality in Diophantine geometry, which was proven by Zhang in \cite[Theorem~1.10]{Zha_Small}, and relates the height of a subvariety $V \subseteq A_N$ to the heights of its points.
More precisely, Zhang's theorem implies that
\begin{equation} \label{ineq:zhang}
	\mu_2(V) \leq \frac{h_2(V)}{\deg(V)} \leq (1 + \dim V) \mu_2(V),
\end{equation} 
where $\mu_2(V)$ denotes the \textit{essential minimum} of $V$ with respect to the Faltings height $h_2$, which is defined to be the infimum of all the real numbers $\theta \in \mathbb{R}_{\geq 0}$ such that the subset $\{ P \in V(\overline{\mathbb{Q}}) \, \colon \, h_2(P) \leq \theta \}$ is Zariski dense in $V$.

\subsubsection*{Differences of heights}

To conclude this subsection, we will recall an explicit inequality 
between the Faltings height $h_2(P)$ and the Néron-Tate height $\hat{h}(P)$ of a point $P \in A_N(\overline{\QQ})$, which was proven in \cite[Proposition~3.2]{CVV_Forum}. 
More precisely, if $E$ is an elliptic curve defined over $\overline{\QQ}$ by the Weierstrass equation $E \colon y^2 = x^3 + A x + B$, we set
\begin{equation} \label{eq:constants_height_comparison}
	\begin{aligned}
		c_1(E) &:= \frac{h(A) + h(B)}{2} + \frac{h(\Delta(E)) + h_\infty(j(E))}{4} + \frac{h(j(E))}{8} + 3.724  \\
		c_2(E) &:= \frac{h(A) + h(B)}{2} + \frac{h(\Delta(E)) + h_\infty(j(E))}{4} + 4.015
	\end{aligned}
\end{equation}
where $j(E)$ denotes the $j$-invariant of $E$ and $\Delta(E)$ denotes the discriminant of the Weierstrass equation we fixed above. Moreover, in these formulas, the function $h_\infty$ is taken with respect to the number field $K = \mathbb{Q}(j(E))$.
Finally, we set $c_1(A_N) := c_1(E_1) + \dots + c_1(E_N)$ and we analogously define $c_2(A_N) := c_2(E_1) + \dots + c_2(E_N)$. 
Then, we have that
\begin{equation} \label{ineq:Zimmer}
	-c_1(A_N) \leq h_2(P) - \hat{h}(P) \leq c_2(A_N)
\end{equation}
for every point $P \in A_N(\overline{\QQ})$.
We also recall that, if $E$ is defined over $\QQ$ and $P \in E(\QQ)$, one can take the better constants
\begin{equation*}
	\begin{aligned}
		c_1(E) := \min&\Bigg( \frac{\log(\lvert A \rvert + \lvert B \rvert + 3)}{2} + \frac{\log\lvert \Delta(E) \rvert + \log\max(\lvert j(E) \rvert,1)}{4} + \frac{h(j(E))}{8} + 2.919,\\ &3 h(1 \colon A^{1/2} \colon B^{1/3}) + 4.709 \Bigg) \\
		c_2(E) := \min&\Bigg( \frac{\log(\lvert A \rvert + \lvert B \rvert + 3)}{2} + \frac{\log\lvert \Delta(E) \rvert + \log\max(\lvert j(E) \rvert,1)}{4} + 3.21, \\ &\frac{3}{2} h(1 \colon A^{1/2} \colon B^{1/3}) + 2.427 \Bigg)
	\end{aligned}
\end{equation*}
instead of those defined in \eqref{eq:constants_height_comparison}.

\subsection{An irreducibility criterion for curves}

We are almost ready to prove \cref{thm:curves_intro}. 
Before doing that, we will need to specialize some of the results recalled in the previous paragraphs to the case of curves.
First of all, the following lemma shows how to combine the inequalities \eqref{ineq:zhang} and \eqref{ineq:Zimmer} in order to give an upper bound for the Weil height of some particularly simple varieties.

\begin{lem} \label{lem:zhang_special}
	For every $j \in \{1,\dots,N\}$ and every $Q_j \in E_j(\overline{\mathbb{Q}})$, let 
	\[
	X := E_1 \times \dots \times E_{j-1} \times \{Q_j\} \times E_{j+1} \times \dots \times E_N,
	\] 
	then 
		\begin{equation} \label{ineq:zhang_special}
			h_2(X) \leq N 3^{N-1} (h_2(Q_j) + c_3(A_N)),
		\end{equation}
		where $c_3(A_N) := c_1(A_N) + c_2(A_N)$.
\end{lem}
\begin{proof}
	First of all, observe that $\deg(X) = 3^{N-1}$ and $\mu_2(X) \leq \hat{h}(Q_j) + c_2(A_N)$, because \eqref{ineq:Zimmer} implies that the set $\{ P \in X(\overline{\mathbb{Q}}) \colon h_2(P) \leq \widehat{h}(Q_j) + c_2(A_N) \}$ contains the set
	\[
	\{ P \in X(\overline{\mathbb{Q}}) \colon \widehat{h}(P) \leq \widehat{h}(Q_j) \},
	\] 
	and the latter contains the set 
	$E_1(\overline{\mathbb{Q}})_\text{tors} \times \dots \times E_{i-1}(\overline{\mathbb{Q}})_\text{tors} \times \{Q_j\} \times E_{i+1}(\overline{\mathbb{Q}})_\text{tors} \times \dots \times E_N(\overline{\mathbb{Q}})_\text{tors}$, 
	which is Zariski dense in $X$.
	Therefore, \eqref{ineq:zhang} implies that
	\[ 
		h_2(X) \leq (1+\dim(X)) \deg(X) \mu_2(X) = N 3^{N-1} \mu_2(X) \leq N 3^{N-1} (\hat{h}(Q_j) + c_2(A_N)),
	\]
	which can be combined with \eqref{ineq:Zimmer} to see that
	\[
		h(X) \leq N 3^{N-1} \left(h(Q_j) + c_3(A_N) \right)
	\]
	as wished.
\end{proof}

Now, let us observe that the formula \eqref{eq:degree_preimage}, which expresses the degree of a preimage of a subvariety $V \subseteq A_N$ in terms of the degree of $V$ itself, can be slightly simplified when $V$ is a curve and $[\alpha_1,\dots,\alpha_N] = [\alpha]_j$ for some $\alpha \in \mathbb{Z}$ and $j \in \{1,\dots,N\}$, as the following lemma shows.

\begin{lem}
	Let $C \subseteq A_N$ be an irreducible curve and fix some $j \in \{1,\dots,N\}$ and $\alpha \in \mathrm{End}(E_j)$. Then, we have that
	\begin{equation} \label{eq:degree_preimage_curves}
		\deg([\alpha]_j^{-1}(C)) = d_j + \deg(\alpha) \sum_{i \neq j} d_i
	\end{equation}
	where $d_i := 3 \deg_0(C \cdot (E_1 \times \dots \times E_{i-1} \times \{(0 \colon 1 \colon 0)\} \times E_{i+1} \times \dots \times E_N))$ for every $i \in \{1,\dots,N\}$.
\end{lem}
\begin{proof}
	It suffices to observe that when $C$ is a curve the only subsets $J \subseteq \{1,\dots,N\}$ which provide a non-trivial contribution to \eqref{eq:degree_preimage} are the singletons $J = \{i\}$. Indeed, setting $\alpha_j = \alpha$ and $\alpha_k = 1$ if $k \neq j$, we see that $\prod_{k \not\in \{i\}} \deg(\alpha_k) = \prod_{k \not\in \{i\}} \deg(\alpha_k) = \deg(\alpha)$ for every $i \in \{1,\dots,N\} \setminus \{j\}$, and that $\prod_{k \not\in \{j\}} \deg(\alpha_k) = 1$. 
	Moreover, we have that $\deg_{I_{\{i\}}}(V) = d_i$ for every $i \in \{1,\dots,N\}$, as follows directly from the definition of multiprojective degree \eqref{eq:multiprojective_degree}, combined with the explicit formula \eqref{eq:embedding_divisors} for the embedding divisors associated to the multiprojective embedding given by \eqref{eq:fixed_embedding}.
\end{proof}

We are finally ready to prove \cref{thm:curves_intro}.

\begin{proof}[Proof (of \cref{thm:curves_intro})]
	Combining \cref{trairr} with \cref{lem:factorization}, we see that it suffices to prove that for every $j \in \{1,\dots,N\}$ the curve $[p_j]_j^{-1}(C)$ is irreducible. 
	
	Therefore, let us fix $j \in \{1,\dots,N\}$, and let us suppose by contradiction that $[p_j]_j^{-1}(C)$ is reducible. Then, the number of components of $[p_j]_j^{-1}(C)$ is either $p_j$ or $p_j^2$. Hence, if $C'$ denotes any irreducible component of $[p_j]_j^{-1}(C)$ which has minimal degree, we see from \eqref{eq:degree_preimage_curves} that 
	\begin{equation} \label{eq:degree_component_curve}
		\deg(C') \leq \frac{\deg([p_j]_j^{-1}(C))}{p_j} \leq d_j' \leq p_j \deg(C),
	\end{equation}
	where $d_j' := d_j + p_j \sum_{k \neq j} d_k$.
	
	Now, the height of $C'$ can be bounded using Zhang's inequality \eqref{ineq:zhang}. 
	More precisely, we have that 
	\begin{equation} \label{ineq:essential_preimage}
		\mu_2(C') \leq \mu_2(C) + c_3(E_j) 
	\end{equation}
	where $c_3(E_j) := c_1(E_j) + c_2(E_j)$ is the constant introduced in \cref{lem:zhang_special}. 
	Indeed, this follows by combining \eqref{ineq:Zimmer} with the fact that
	$\widehat{h}(P) \leq \widehat{h}([p_j]_j(P))$ for every point $P \in C(\overline{\QQ})$.
	Therefore,
	\begin{equation}\label{ineq:hc_Weil}
		h_2(C') \le 2 \deg(C') \mu_2(C') \le 2 p_j \deg(C) (\mu_2(C)+c_3(E_j)) \le 2 p_j (h_2(C) + c_3(E_j) \deg(C)),
	\end{equation}
	where the first and third inequality follow from Zhang's inequality \eqref{ineq:zhang}, while the second one follows from \eqref{eq:degree_component_curve} and \eqref{ineq:essential_preimage}.
	
	Let us now fix a set of points $\mathcal{Q} \subseteq C(\overline{\QQ})$ such that the set of Faltings heights $h_2(\mathcal{Q}) \subseteq \mathbb{R}_{\geq 0}$ is unbounded. Such a set surely exists, because any projection of $C$ onto one of the factors of $A_N$ is surjective, since $C$ is transverse.
	Moreover, thanks to the pigeonhole principle, we can assume, up to shrinking $\mathcal{Q}$, that there exists $i \in \{1,\dots,N\}$ such that for every $k \in \{1,\dots,N\}$ and every point $Q = (Q_1,\dots,Q_N) \in \mathcal{Q}$ we have that $h_2(Q_k) \leq h_2(Q_i)$. 
	
	Suppose now that $i \neq j$. Then, the arithmetic Bézout theorem yields an upper bound for the Faltings height of the $i$-th coordinate $Q_i \in E_i(\overline{\QQ})$ of a point $Q = (Q_1,\dots,Q_N) \in \mathcal{Q}$.
	More precisely, fix such a point $Q \in \mathcal{Q}$. 
	Then, there exists a point $Q_j' \in E_j(\overline{\QQ})$ such that $[p_j](Q_j') = Q_j$ and $Q' \in C'(\overline{\QQ})$, where $Q' := (Q_1,\dots,Q_{j-1},Q_j',Q_{j+1},\dots,Q_N)$.
	Note that \cref{trairr} implies that $C'$ is transverse, because it is an irreducible component of the preimage of a transverse curve. Therefore, we see that $\{Q'\}$ is a component of the intersection $X_j' \cap C'$, where $X_j' := E_1 \times \dots \times E_{j-1} \times \{Q_j'\} \times E_{j+1} \times \dots \times E_N$.
	We apply the arithmetic Bézout theorem \eqref{ineq:bezout} to the intersection $X_j' \cap C'$, obtaining 
	\[
	h_2(Q_j') + \sum_{k \neq j} h_2(Q_k) = h_2(\{Q'\}) \leq h_2(X_j') \deg(C') + h_2(C') \deg(X_j') + c_4 \deg(C') \deg(X_j'),
	\]
	where $c_4 := c_0(1,N-1,3^{N-1})$.
	Combining this with \eqref{ineq:zhang_special}, we see that
	\begin{equation} \label{ineq:applied_bezout_1}
		h_2(Q_j') + \sum_{k \neq j} h_2(Q_k) \leq 3^{N-1} (N h_2(Q_j') \deg(C') + h_2(C') + (c_3(A_N) + c_4) \deg(C')),
	\end{equation}
	because $\deg(X_j') = 3^{N-1}$.
	Moreover, \eqref{ineq:Zimmer} implies that
	\begin{equation} \label{ineq:Weil_height_division}
		h_2(Q_j') \leq \hat{h}(Q_j') + c_2(E_j)  = 
		\frac{\hat{h}(Q_j)}{p_j^2} + c_2(E_j) \leq
		\frac{h_2(Q_j)}{p_j^2} + c_3(E_j).
	\end{equation}
	Finally, we have that
	\begin{equation} \label{ineq:i_neq_j}
		h_2(Q_i) \leq h_2(Q_j') + \sum_{k \neq j} h_2(Q_k),
	\end{equation}
	because we are assuming that $i \neq j$, and we also know that
	\begin{equation} \label{ineq:Q}
		h_2(Q_j) \leq h_2(Q_i)
	\end{equation}
	because $Q \in \mathcal{Q}$.
	Combining the upper bound given by \eqref{ineq:applied_bezout_1} together with the inequalities \eqref{eq:degree_component_curve}, \eqref{ineq:hc_Weil}, \eqref{ineq:Weil_height_division}, \eqref{ineq:i_neq_j} and \eqref{ineq:Q}, we obtain
	\begin{equation} \label{ineq:bound_ineqj}
		3^{1-N} h_2(Q_i) \leq N \frac{d_j'}{p_j^2} h_2(Q_i) + (2 (\mu_2(C) + c_3(A_N)) + c_4) d_j'.
	\end{equation}
	Therefore, the height of $Q_i$ is uniformly bounded above, because we have by assumption that
	\[
	p_j^2 \geq (\deg(C) N 3^{N-1})^2 > d_j' N 3^{N-1}.
	\]
	This contradicts the fact that $h_2(Q) \leq N h_2(Q_i)$ is unbounded, and allows us to conclude that the preimage $[p_j]_j^{-1}(C)$ cannot be reducible whenever $i \neq j$.
	
	Suppose, on the other hand, that $i = j$, and fix a point $Q = (Q_1,\dots,Q_N) \in \mathcal{Q}$. Then, the inequality \eqref{ineq:applied_bezout_1} still holds true, but it is not sufficient any more to get to a contradiction. Instead, we will need to combine \eqref{ineq:applied_bezout_1} with another application of the arithmetic Bézout theorem. More precisely, since $C$ is transverse, the point $\{Q\}$ is a component of each of the intersections $X_l \cap C$, where $l \in \{1,\dots,N\} \setminus \{i\}$ and $X_l$ is defined as $X_l := E_1 \times \dots \times E_{l-1} \times \{Q_l\} \times E_{l+1} \times \dots \times E_N$.
	Therefore, we see that 
	\begin{equation}
		h_2(Q_l) + h_2(Q_i) \leq h_2(\{Q\}) \leq 3^{N-1} (N h_2(Q_l) \deg(C) + h_2(C) + (c_3(A_N) + c_4) \deg(C)),
	\end{equation}
	where the second inequality follows from a combination of \eqref{ineq:bezout} and \eqref{ineq:zhang_special}.
	This implies that there exists a constant $c_5 \in \mathbb{R}_{> 0}$, depending on $C$, such that 
	\begin{equation} \label{ineq:extra_bound_ieqj}
		h_2(Q_j) = h_2(Q_i) \leq (3^{N-1} N \deg(C) - 1) h_2(Q_l) + c_5
	\end{equation}
	for every $l \in \{1,\dots,N\} \setminus \{i\}$.
	In particular, we see that for every $l \in \{1,\dots,N\}$ the set 
	\[
	\{h_2(Q_l) \colon Q = (Q_1,\dots,Q_N) \in \mathcal{Q}\} \subseteq \mathbb{R}_{\geq 0}
	\] 
	is unbounded. 
	Now, choose any $l \in \{1,\dots,N\} \setminus \{i\}$, and we see that an inequality similar to \eqref{ineq:bound_ineqj} holds true.
	More precisely, we have that $h_2(Q_l) \leq h_2(Q_j') + \sum_{k \neq j} h_2(Q_k)$, which allows us to see that
	\begin{equation} \label{ineq:bound_ieqj}
		3^{1-N} h_2(Q_l) \leq N \frac{d_j'}{p_j^2} h_2(Q_j) + (2 (\mu(C) + c_3(A_N)) + c_4) d_j'
	\end{equation}
	by combining once again \eqref{ineq:applied_bezout_1} together with the inequalities \eqref{eq:degree_component_curve}, \eqref{ineq:hc_Weil} and \eqref{ineq:Weil_height_division}.
	Combining this with \eqref{ineq:extra_bound_ieqj} guarantees that
	\[
	3^{1-N} h_2(Q_l) \leq N \frac{d_j'}{p_j^2} (3^{N-1} N \deg(C) - 1) h_2(Q_l) + (2 (\mu(C) + c_3(A_N)) + c_4) d_j' + N \frac{d_j'}{p_j^2} c_5. 
	\]
	As before, this allows us to conclude that $h_2(Q_l)$ is bounded, because 
	\[
	p_j^2 \geq (\deg(C) N 3^{N-1})^2 > d_j' N 3^{N-1} (3^{N-1} N \deg(C) - 1),
	\] 
	by assumption. Since this contradicts what we have shown before, $[p_j]_j^{-1}(C)$ must be irreducible even when $i = j$, as we wanted to prove.
\end{proof}

\begin{remark} \label{rmk:arithmetic_not_generalizing}
	We remark that a generalisation of this proof to a transverse variety $V \subseteq A_N$ is not directly possible. 
	Indeed, to prove something analogous to \cref{thm:curves_intro} it would be sufficient to prove that, given a prime $p$, the preimage $[p,\dots,p,1,\dots,1]^{-1}(V)$ is irreducible when $p$ is sufficiently big, where 
	the diagonal endomorphism $[p,\dots,p,1,\dots,1]$ has $\dim(V)$ components equal to $p$.
	To do so, one would be tempted to consider a point $Q' \in V'$, where $V'$ is some irreducible component of $[p,\dots,p,1,\dots,1]^{-1}(V)$.
	Then, $Q'$ would be an irreducible component of the intersection \[V' \cap (\{Q'\} \times A_{N - \dim(V)}),\] to which one could 
	apply the arithmetic Bézout theorem \eqref{ineq:bezout}.
	However, the gain one obtains by considering the height of $Q'$ is just $\frac{1}{p^2}$, which is not sufficient to overtake the degree of $V'$, which can only be bounded by $p^{2 (\dim(V) - 1)} \deg(V)$.
\end{remark}

The previous remark prompted us to use a more geometric approach to study the transversality of preimages of higher dimensional varieties. 
This is also a hint for the difficulties that one encounters when trying to extend to higher dimensional varieties the methods used in the proof of the Torsion Anomalous Conjecture for curves.

\section{Preimages of transverse subvarieties: a geometric approach}
\label{sec:general}

The aim of this section is to give a geometric proof of the following result, which guarantees that the preimage by suitable group homomorphisms of a transverse subvariety $V \subseteq A_N$ remains transverse. 

\begin{thm} \label{thm:stronger_general_varieties}
	Let $V \subseteq A_N$ be a transverse subvariety, and let 
	\[(\alpha_1,\dots,\alpha_N) \in \prod_{j=1}^N \mathrm{End}(E_j)\] be a tuple of isogenies.
	Moreover, suppose that for every $j \in \{1,\dots,N\}$ there exists a subset $J_j \subseteq \{1,\dots,N\}$ of cardinality $\dim(V)$ such that $j \in J_j$ and
		\[
		\gcd(\deg(\alpha_j),\dim(V)! \deg_{I_{J_j}}(V)) = 1,
		\]
		where $I_{J_j} = (i_{1,j},\dots,i_{N,j})$ is the tuple defined by setting $i_{k,j} = 1$ if $k \in J_j$, and $i_{k,j} = 0$ otherwise, according to the notation introduced in \cref{lem:degree_preimage}.
		Then, the preimage \[[\alpha_1,\dots,\alpha_N]^{-1}(V)\] is transverse.
\end{thm}

This theorem is stronger than our main theorem in the introduction, as specified in the following lemma.
	
	\begin{lem} \label{lem:varieties_strong_to_weak}
	\cref{thm:stronger_general_varieties} implies \cref{thm:varieties_intro}.
	\end{lem}
	\begin{proof} We observe that if we have a tuple $(\alpha_1,\dots,\alpha_j) \in \prod_{j=1}^N \mathrm{End}(E_j)$ such that for every index $j \in \{1,\dots,N\}$ and every prime $p \mid \deg(\alpha_j)$ we have that $p > \dim(V)! \deg(V)$, then in particular $p > \dim(V)! \deg_I(V)$ for every $I \subseteq \{1,\dots,N\}$ such that $\lvert I \rvert = \dim(V)$, thanks to \eqref{eq:Hilbert_Samuel}. Therefore, we see that $p \nmid \dim(V)! \deg_I(V)$ for any $I \subseteq \{1,\dots,N\}$ such that $\lvert I \rvert = \dim(V)$, which clearly implies that the hypotheses of \cref{thm:stronger_general_varieties} are satisfied.
\end{proof}

We are finally ready to prove \cref{thm:stronger_general_varieties}, using most of the results that we proved in \cref{sec:preliminaries}.

\begin{proof}[Proof (of \cref{thm:stronger_general_varieties})]
	Let us consider the endomorphisms $f_j = [\beta_{1,j},\dots,\beta_{N,j}]$, where $\beta_{k,j} := \alpha_j$ if $k \in J_j$, and $\beta_{k,j} = 1$ otherwise.
	Then, thanks to \cref{lem:factorization}, it suffices to show that each of the varieties $f_1^{-1}(V),\dots,f_N^{-1}(V)$ is irreducible. Indeed, if this happens, then for every $j \in \{1,\dots,N\}$ the variety \[[\alpha_j]^{-1}(V) = [\beta_{1,j},\dots,\beta_{j-1,j},1,\beta_{j+1,j},\dots,\beta_{N,j}](f_j^{-1}(V))\] is irreducible, because $\beta_{j,j} = \alpha_j$.
	
	Therefore, let us fix any $j \in \{1,\dots,N\}$, and let us show that $f_j^{-1}(V)$ is irreducible. 
	To do so, let us suppose by contradiction that the variety $f_j^{-1}(V)$ is reducible, and let $d > 1$ be the number of its irreducible components. Then, $d$ divides $\deg(f_j) = \deg(\alpha_j)^{\dim(V)}$.
	Moreover, \cref{degcomponenti} implies that $\deg(f_j^{-1}(V)) = d \deg(W)$, where $W$ is any irreducible component of $f_j^{-1}(V)$.
	This shows that $\deg(\alpha_j)$ and $\deg(f_j^{-1}(V))$ are not coprime.
	On the other hand, \cref{lem:degree_preimage} implies that
	\[
		\deg(f_j^{-1}(V)) - \dim(V)! \deg_{I_{J_j}}(V) = \dim(V)! \sum_{J \neq J_j} \left( \prod_{k \not\in J} \deg(\beta_{k,j}) \right) \deg_{I_J}(V),
	\]
	where the sum on the right hand side runs over all the subsets $J \subseteq \{1,\dots,N\}$ such that $\lvert J \rvert = \dim(V)$ and $J \neq J_j$. 
	These two conditions show that for every such $J$ there exists $k \not\in J$ such that $\beta_{k,j} = \alpha_j$, as one sees from the definition of $f_j$. 
	Therefore, we see that
	\begin{equation} \label{eq:divisibility}
		\deg(\alpha_j) \mid \deg(f_j^{-1}(V)) - \dim(V)! \deg_{I_{J_j}}(V),
	\end{equation}
	which implies that $\deg(\alpha_j)$ and $\dim(V)! \deg_{I_{J_j}}(V)$ are not coprime.
	However, this contradicts our assumptions, and allows us to conclude that $f_j^{-1}(V)$ must be irreducible, as we wanted to show.
\end{proof}

In particular, we see that \cref{thm:stronger_general_varieties} implies the following result for endomorphisms which are multiples of the identity.

\begin{cor} \label{cor:multiple_identity}
	Let $V \subseteq A_N$ be a transverse variety, and let $n \in \mathbb{Z}$ be an integer such that for every $j \in \{1,\dots,N\}$ there exists a tuple $I = (i_1,\dots,i_N) \in \mathbb{N}^N$ with $i_j = 1$ such that
	\[\gcd(n,\dim(V)! \deg_I(V)) = 1.\]
	Then, the preimage $[n,\dots,n]^{-1}(V)$ is transverse.
	
	In particular, if $p \in \mathbb{Z}$ is a prime such that $p \nmid \dim(V)!$ and 
	\begin{equation} \label{eq:assumption_corollary_diagonal}
		p \nmid \gcd\{ \deg_I(V) \colon  I = (i_1,\dots,i_N) \in \{0,1\}^N, \ i_1 + \dots + i_N = \dim(V), \ i_j = 1\}
	\end{equation} 
	for every $j \in \{1,\dots,N\}$, then $[p,\dots,p]^{-1}(V)$ is transverse. 
\end{cor}
\begin{proof}
	The first part of this corollary is precisely obtained by setting $\alpha_1 = \dots = \alpha_N = n$ in \cref{thm:stronger_general_varieties}, so there is nothing to prove.
	For the second part, suppose by contradiction that setting $\alpha_1 = \dots = \alpha_N = p$ yields a tuple of isogenies which does not satisfy the assumptions of \cref{thm:stronger_general_varieties}.
	This implies necessarily that $p \mid \deg_I(V)$ for every tuple $I = (i_1,\dots,i_N) \in \{0,1\}^N$ such that $i_1 + \dots + i_N = \dim(V)$, because for each of these tuples there exists $j \in \{1,\dots,N\}$ such that $i_j = 1$, since $\dim(V) \geq 1$.
	However, this divisibility property contradicts our assumption \eqref{eq:assumption_corollary_diagonal}, and this allows us to conclude that the tuple of isogenies $(p,\dots,p)$ satisfies the assumptions of \cref{thm:stronger_general_varieties}, which implies that $[p,\dots,p]^{-1}(V)$ is transverse.
\end{proof}

Finally, we provide a corollary to \cref{thm:stronger_general_varieties} for curves.

\begin{cor} \label{cor:curves_from_varieties}
	Let $C \subseteq A_N$ be a transverse curve, and let $(\alpha_1,\dots,\alpha_N) \in \prod_{j=1}^N \mathrm{End}(E_j)$ be a tuple of isogenies. Moreover, suppose that for every $j \in \{1,\dots,N\}$ we have that
	\[
		\gcd(\deg(\alpha_j),\deg_j(C)) = 1
	\]
	where $\deg_1(C) := \deg_{(1,0,\dots,0)}(C),\dots,\deg_N(C) := \deg_{(0,\dots,0,1)}(C)$. Then, $[\alpha_1,\dots,\alpha_N]^{-1}(C)$ is transverse.
\end{cor}

\section{Examples and applications}
\label{sec:examples}

The aim of this section is to provide some applications of our main results.

\subsection{Transversality of specific subvarieties}

Let us see how our main results allow one to prove that some specific subvarieties of a product of elliptic curves are irreducible.

First of all, we give an example for $N = 2$. More precisely, for every $n \in \mathbb{Z}_{\geq 1}$ we consider the curve $C_n \subseteq E_1 \times E_2$ which is the projective closure of the affine curve
\begin{equation} \label{eq:transverse_family}
	C_n^\circ \colon \begin{cases}
		y_2 &= x_1^n \\
		y_1^2 &= x_1^3 + A_1 x_1 + B_1 \\
		y_2^2 &= x_2^3 + A_2 x_2 + B_2
	\end{cases}
\end{equation}
inside $\mathbb{A}^2 \times \mathbb{A}^2$.
Then, \cite[Theorem~6.2]{CVV_Forum} shows that $C_n$ is transverse, and computes that 
\[
\begin{aligned}
	\deg_{(1,0)}(C_n) &= 9 \\
	\deg_{(0,1)}(C_n) &= 6 n,
\end{aligned}
\] which implies that $\deg(C_n) = 6 n + 9$. 

Let us see how one can make explicit the equations of the preimages of $C_n$.
To this end, we recall that for every elliptic curve $E$ embedded in $\mathbb{P}^2$ via a short Weierstrass equation $y^2 z = x^3 + A x z^2 + B z^3$, and every $\alpha \in \mathbb{Z}$, there exist three polynomials $r_{\alpha}, s_\alpha, t_\alpha \in \mathbb{Q}[x,z,A,B]$, which are homogeneous in $x$ and $z$, such that
\begin{equation} \label{eq:scalar_multiplication}
	[\alpha]_{E}(P) = \left( r_\alpha(x,z,A,B) t_\alpha(x,z,A,B) \colon s_\alpha(x,z,A,B) y \colon t_\alpha(x,z,A,B)^3 z \right)
\end{equation}
for every point $P = (x \colon y \colon z) \in E \setminus \ker([\alpha]_{E})$, as shown in \cite[Chapter~13, Section~9]{Husemoller_2004}, and \cite[Lecture~4, Section~4.4]{Sutherland_2022}. 
Moreover, the polynomial $t_\alpha$ is the homogenized version of the $\alpha$-th division polynomial of $E$, and we always have that \[\gcd(r_\alpha,s_\alpha) = \gcd(s_\alpha,t_\alpha) = 1,\] whereas $\gcd(r_\alpha,t_\alpha) = 1$ if and only if $2 \nmid \alpha$.
On the other hand, when $2 \mid \alpha$ we have that 
\[
	\begin{aligned}
		r_\alpha &= (x^3 + A x z^2 + B z^3) \tilde{r}_\alpha \\
		t_\alpha &= (x^3 + A x z^2 + B z^3) \tilde{t}_\alpha
	\end{aligned}
\]
where $\tilde{r}_\alpha, \tilde{t}_\alpha \in \mathbb{Q}[x,z,A,B]$ are homogeneous in $x$ and $z$, and we have that $\gcd(\tilde{r}_\alpha,\tilde{t}_\alpha) = 1$.
Using these facts, we see that \eqref{eq:scalar_multiplication} still holds for every point $(x \colon y \colon z) \in \ker([\alpha]_{E})$, unless $2 \mid \alpha$ and $P \in \ker([2]_{E})$, in which case the morphism appearing in \eqref{eq:scalar_multiplication} is not well defined.
To avoid this issue, one can multiply the three polynomials appearing on the right hand side of \eqref{eq:scalar_multiplication} by $y z$, and then divide everything by $x^3 + A x z^2 + B z^3$, to obtain the following formula
\begin{equation} \label{eq:scalar_multiplication_even}
	[\alpha]_{E}(P) = (r_\alpha(x,z,A,B) \tilde{t}_\alpha(x,z,A,B) y z \colon s_\alpha(x,z,A,B) \colon \tilde{t}_\alpha(x,y,A,B) t_\alpha(x,y,A,B)^2 y z^2),
\end{equation}
which is valid for every point $P := (x \colon y \colon z) \in E$.

Substituting a dehomogenized version of the equations \eqref{eq:scalar_multiplication} and \eqref{eq:scalar_multiplication_even} inside \eqref{eq:transverse_family}, we see that for every pair of \emph{odd} integers $\alpha_1,\alpha_2 \in \mathbb{Z}$, the preimage $[\alpha_1,\alpha_2]^{-1}(C_n)$ is given by the projective closure of the affine curve $([\alpha_1,\alpha_2]^{-1}(C_n))^\circ = V \cap (E_1^\circ \times E_2^\circ) \subseteq \mathbb{A}^2 \times \mathbb{A}^2$, where $E_i^\circ \colon y_i^2 = x_i^3 + A_i x_i + B_i$ for every $i \in \{1,2\}$, while
\[
	V \colon y_2 \, s_{\alpha_2}(x_2,1,A_2,B_2) \, t_{\alpha_1}(x_1,1,A_1,B_1)^{2 n} = r_{\alpha_1}(x_1,1,A_1,B_1)^n \, t_{\alpha_2}(x_2,1,A_2,B_2)^3.
\]
On the other, hand, if for example $\alpha_1$ is even and $\alpha_2$ is odd, an affine model for the curve $[\alpha_1,\alpha_2]^{-1}(C_n)$ is given by $([\alpha_1,\alpha_2]^{-1}(C_n))^\circ = W \cap (E_1^\circ \times E_2^\circ) \subseteq \mathbb{A}^2 \times \mathbb{A}^2$, where
\[
	\begin{aligned}
		W &\colon y_2 s_{\alpha_2}(x_2,1,A_2,B_2) t_{\alpha_1}(x_1,1,A_1,B_1)^n \tilde{t}_{\alpha_1}(x_1,1,A_1,B_1)^{n} \\ &= \tilde{r}_{\alpha_1}(x_1,1,A_1,B_1)^n t_{\alpha_2}(x_2,1,A_2,B_2)^3.
	\end{aligned}
\]

In particular, the equations describing these preimages can be quite complicated, because the size of the coefficients of the polynomials $r_\alpha, s_\alpha$ and $t_\alpha$ grows with respect to $\lvert \alpha \rvert$.
For example, we have that
\[
	\begin{aligned}
		r_2(x,1,A,B) &= x^7 - A x^5 - 7 B x^4 - A 2 x^3 - 10 A B x^2 + (A^3 - 8 B^2) x + A^2 B \\
		s_2(x,1,A,B) &= x^9 + 6 A x^7 + 21 B x^6 + 21 A B x^4 + (12 B^2 -6 A^3) x^3 - 9 A^2 B x^2 \\ &- (A^4 + 12 A B^2) x -A^3 B - 8 B^3 \\
		t_2(x,1,A,B) &= 2 (x^3 + A x + B)
	\end{aligned}
\]
whereas
\[
	\begin{aligned}
		r_3(x,1,A,B) &= x^9 - 12 A x^7 - 96 B x^6 + 30 A^2 x^5 - 24 A B x^4 + (36 A^3 + 48 B^2) x^3 + 48 A^2 B x^2\\ 
		&+ (9 A^4 + 96 A B^2) x + 8 A^3 B + 64 B^3 \\
		s_3(x,1,A,B) &= x^{12} + 22 A x^{10} + 220 B x^9 - 165 A^2 x^8 - 528  A B x^7 - (92 A^3 + 1776 B^2) x^6 \\
		&+ 264 A^2 B x^5 - (185 A^4 + 960 A B^2) x^4 -(80 A^3 B - 320 B^3) x^3 \\ &- (90 A^5 + 624 A^2 B^2) x^2
		- (132 A^4 B + 896 A B^3) x - 3 A^6 - 96 A^3 B^2 - 512 B^4 \\
		t_3(x,1,A,B) &= 3 x^4 + 6 A x^2 + 12 B x - A^2.
	\end{aligned}
\]
Therefore, we see that checking whether the curves $[\alpha,1]^{-1}(C_n)$ and $[1,\alpha]^{-1}(C_n)$ are irreducible can be difficult, even with the help of a computer. 

On the other hand, \cref{thm:curves_intro} shows that if $p \in \mathbb{Z}$ is a prime such that $\lvert p \rvert \geq 6 (6 n + 9)$ then the two curves $[p,1]^{-1}(C_n)$ and $[1,p]^{-1}(C_n)$ are transverse (and in particular irreducible), because we have that $\deg(C_n) = 6 n + 9$, as we recalled above.
These results can in fact be improved using \cref{cor:curves_from_varieties}. 
More precisely, since $\deg_{(1,0)}(C_n) = 9$ and $\deg_{(0,1)}(C_n) = 6 n$ we see that $[\alpha,1]^{-1}(C_n)$ is transverse whenever $3 \nmid \alpha$, whereas $[1,\alpha]^{-1}(C_n)$ is transverse if $\gcd(\alpha,6n) = 1$. 
The following example gives an explicit illustration of this transversality criterion.

\begin{example} \label{ex:irreducible_preimage}
	Let $A_1 = A_2 = 0$ and $B_1 = B_2 = 1$, so that $E_1 = E_2$ is an elliptic curve of conductor $36$. Then, the previous considerations imply that the projective closure of the affine curve
	\begin{equation} \label{eq:preimage_C3}
		([2,1]^{-1}(C_3))^\circ \colon
		\begin{cases}
			y_2 (x_1^3+1)^3 = (x_1^4 - 8 x_1)^3 \\
			y_1^2 = x_1^3 + 1 \\
			y_2^2 = x_2^3 + 1
		\end{cases}
	\end{equation}
	is irreducible inside $(\mathbb{P}^2)^2$.
	Analogously, we see that the curve $[1,5]^{-1}(C_3)$, which is the projective closure of the affine curve $V^\circ \cap (E_1^\circ \times E_2^\circ)$, where $E_i^\circ \colon y_i^2 = x_i^3 + 1$ for $i \in \{1,2\}$, and 
	\[
		 \begin{aligned}
		 	V^\circ &\colon y_2 (x_2^{36}                                                                   
		 	+ 4692\*x_2^{33}
		 	- 884544\*x_2^{30}
		 	+ 1880320\*x_2^{27}
		 	- 94222080\*x_2^{24} \\
		 	&- 1437769728\*x_2^{21}
		 	- 3534606336\*x_2^{18}
		 	- 8883929088\*x_2^{15}
		 	- 6868500480\*x_2^{12} \\
		 	&- 1853358080\*x_2^9
		 	- 497025024\*x_2^6
		 	- 742391808\*x_2^3
		 	+ 16777216)^3 = \\
		 	&= x_1^3 (x_2^{12} + 76 x_2^9 - 48 x_2^6 - 320 x_2^3 - 256/5)^3
		 \end{aligned}
	\]
	is irreducible in $(\mathbb{P}^2)^2$.
\end{example}

\begin{algorithm}[t]
	\textbf{Input}: A subvariety $V \subseteq \mathbb{A}^4$, two elliptic curves $E_1,E_2$ and a diagonal endomorphism $f \colon E_1 \times E_2 \to E_1 \times E_2$.
	\begin{lstlisting}
		A.<x1,y1,x2,y2> = AffineSpace(QQ,4) 
		R.<x,y,z> = PolynomialRing(QQ)
		A.inject_variables()
		def Preimage(V,E1,E2,f):
		E1 = E1.short_weierstrass_model()
		E2 = E2.short_weierstrass_model()
		e1 = (R(E1.defining_polynomial())).subs(x=x1,y=y1,z=1)
		e2 = (R(E2.defining_polynomial())).subs(x=x2,y=y2,z=1)
		f1 = [h.subs(x=x1,y=y1) for h in E1.scalar_multiplication(f[0]).rational_maps()]
		f2 = [h.subs(x=x2,y=y2) for h in E2.scalar_multiplication(f[1]).rational_maps()]
		L = V.defining_polynomials()
		M = [e1,e2]
		for l in L:
		M.append((l.subs({x1:f1[0],y1:f1[1],x2:f2[0],y2:f2[1]}).numerator()))
		return A.subscheme(M)
		S = Preimage(V,E1,E2,f).irreducible_components()
	\end{lstlisting}
	\textbf{Output:} The set $S$ of affine models of the irreducible components of $f^{-1}(C)$, where \\ $C := \overline{V} \cap (E_1 \times E_2)$ is the intersection of the closure of $V$ inside $(\mathbb{P}^2)^2$, denoted $\overline{V}$, with $E_1 \times E_2$.
	\vspace{0.2cm}
	\caption{\textsc{SageMath} code to compute the irreducible components of the preimage of a curve $C \subseteq A_2$.}
	\label{alg:irreducibility}
\end{algorithm}

\subsubsection*{Computational aspects}

Since these curves are defined by explicit equations, one could also try to check their irreducibility using a software such as \textsc{SageMath}. This can be easily done using the code portrayed in \cref{alg:irreducibility}. However, even an irreducibility check in this simple case turns out to be very expensive from a computational point of view.

\begin{remark}
	Let us note that the \textsc{SageMath} command \texttt{rational\_maps()} applied to the scalar multiplications $f_1$ and $f_2$ in \cref{alg:irreducibility} allows one to obtain the canonical, simplified form of the isogenies $f_1$ and $f_2$, which corresponds to our choice of the polynomials $r_\alpha, s_\alpha$ and $t_\alpha$. 
	On the other hand, if one uses the other natural command \texttt{as\_morphism().defining\_polynomials()}, the projective equations that one obtains are not reduced to the lowest terms.
	In particular, writing \cref{alg:irreducibility} with such a command would yield preimages which contain always components of the form $\ker(f_1) \times E_2$ and $E_1 \times \ker(f_2)$, which clearly cannot happen.
\end{remark}

\subsubsection*{Higher dimensions}
To conclude, let us note that one can consider more generally curves $C \subseteq A_N$ when $N \geq 3$. 
For instance, \cite[Theorem~2]{Viada_2021} shows that for every family of non-constant polynomials $p_1,\dots,p_{N-1} \in \overline{\mathbb{Q}}[t]$, the curve defined in $A_N$ by the affine equations
\[
	\begin{cases}
		y_1 = p_1(x_2) \\
		\dots \\
		y_{N-1} = p_{N-1}(x_N)
	\end{cases}
\]
is transverse. Therefore, carrying out a computation analogous to the one that we performed above, we can apply our \cref{thm:curves_intro,thm:varieties_intro} to this family of curves, in order to show that some of their preimages under diagonal endomorphisms are irreducible.

\subsection{Bounding the height of intersections}
Our main results can be used to bound explicitly the Faltings height of an irreducible component of intersections of the form $\phi^{-1}(C) \cap B$, where $B \subseteq A_N$ is a subgroup, $C \subseteq A_N$ is a transverse curve and $\phi \colon A_N \to A_N$ is a diagonal endomorphism such that $\phi(B) = B$.

More precisely, assume to know the degree of $\phi^{-1}(C)$. 
For example, one might know its equations.  
Let $C_0$ be an irreducible component of $\phi^{-1}(C)$. 
Then, the arithmetic Bézout theorem \eqref{ineq:bezout} implies that
\begin{equation} \label{ineq:bezout_trivial}
\begin{aligned}
	h_2(C_0 \cap B) &\leq \left( \begin{aligned}
		\deg(C_0) h_2(B) &+ \deg(B) h_2(C_0) \\ &+ c_0(1,d,3^N-1) \deg(C_0) \deg(B)
	\end{aligned} \right)
	\\ &\le \left( \begin{aligned}
		\deg(\phi^{-1}(C)) h_2(B) &+ \deg(B) h_2(\phi^{-1}(C)) \\ &+ c_0(1,d,3^N-1) \deg(\phi^{-1}(C)) \deg(B)
	\end{aligned} \right),
\end{aligned}
\end{equation}
where $d = \dim(B)$. 
Unfortunately without further knowledge on the number of components of $\phi^{-1}(C)$, for the degree of $C_0$, one has to use the trivial bound $\deg(C_0) \le \deg(\phi^{-1}(C))$. 
On the other hand, only the irreducibility of $\phi^{-1}(C)$ ensures that $C_0$ is stabilized by $\ker(\phi)$. 
So without our theorem it is not easy to improve  \eqref{ineq:bezout_trivial}.

However, \cref{thm:stronger_general_varieties}  shows that in a large number of cases, $\phi^{-1}(C)$ is irreducible. 
Moreover it is stable under the action of $\ker(\phi)$, because we assumed that $\phi(B) = B$. 
Thus $\phi^{-1}(C) \cap B$ consists of $\deg(\phi)$ irreducible components of the same height. 
Therefore, by the arithmetic Bézout theorem we obtain 
\begin{equation} \label{ineq:Bezout_improved}
		\deg(\phi) h_2(P) \leq \left( \begin{aligned} \deg(\phi^{-1}(C)) h_2(B) &+ \deg(B) h_2(\phi^{-1}(C)) \\ &+ c_0(1,d,3^N-1) \deg(\phi^{-1}(C)) \deg(B) \end{aligned} \right),
\end{equation}
for every $P \in (\phi^{-1}(C) \cap B)(\overline{\QQ})$ which can considerably improve \eqref{ineq:bezout_trivial} when $\phi$ has a big degree.
Note that here it is central the fact that we explicitly know the degree and irreducibility of $\phi^{-1}(C)$, and not only (or, in fact, not necessarily) the degree of $C$.

\subsection{Lower bounds for the essential minima of images}

Our main theorem can be applied also to get new lower bounds for the essential minima, with respect to the Néron-Tate height, of the images of certain curves $C \subseteq A_N$ by some endomorphisms $\Phi \colon A_N \to A_N$. In particular, we obtain these bounds by applying a result of Galateau \cite{Galateau_2010} in an indirect way, which requires to combine it with our \cref{thm:curves_intro,thm:stronger_general_varieties}. As we will show, this allows us to improve the lower bounds obtained from a direct application of Galateau's result.  

More precisely, fix two integers $r, N \geq 2$, a product of elliptic curves $A_N := E_1 \times \dots \times E_N$, and an endomorphism $\Phi \colon A_N \to A_N$ which admits the matrix representation
\begin{equation} \label{eq:Phi}
	\Phi = \left( \begin{array}{c:c}
		\alpha \cdot \mathrm{Id}_r & L \\\hdashline[2pt/2pt]
		\mathbf{0} & \mathrm{Id}_{N-r}
	\end{array} \right),
\end{equation}
where $\alpha \in \mathbb{Z}$ and $L = (\ell_{i,j}) \colon E_{r+1} \times \dots \times E_N \to E_1 \times \dots \times E_r$ is a morphism of abelian varieties with components $\ell_{i,j} \colon E_j \to E_i$, for $j \in \{r+1,\dots,N\}$ and $i \in \{1,\dots,r\}$.
Moreover, let us suppose that $\alpha^2 \geq d_L := \max_{i,j}(\deg(\ell_{i,j}))$. 
Then, if we define the automorphism $\Psi \colon A_N \to A_N$ as
\[
	\Psi := \left( \begin{array}{c:c}
		\mathrm{Id}_r & L \\\hdashline[2pt/2pt]
		\mathbf{0} & \mathrm{Id}_{N-r}
	\end{array} \right),
\]
we have the obvious relation $\Phi \circ f = [\alpha,\dots,\alpha] \circ \Psi$, where $f := [1,\dots,1,\alpha,\dots,\alpha]$ is the diagonal endomorphism with the first $r$ entries equal to one and the last $N-r$ entries equal to $\alpha$. In particular, we have that
\[
	\Phi(C) = [\alpha,\dots,\alpha](\Psi(f^{-1}(C))).
\] 
Since the Néron-Tate height is a quadratic form, this implies that
\begin{equation} 
	\label{eq:essential_minimum_image}
	\hat{\mu}(\Phi(C)) = \alpha^2 \hat{\mu}(\Psi(f^{-1}(C))),
\end{equation} 
where $\hat{\mu}$ denotes the essential minimum with respect to the Néron-Tate height $\hat{h} \colon A_N(\overline{\QQ}) \to \mathbb{R}$, which is defined as
\[
	\hat{\mu}(V) := \inf\{\theta \in \mathbb{R}_{> 0} \colon \{P \in V(\overline{\QQ}) \colon \hat{h}(P) \leq \theta\} \ \text{is Zariski dense in} \ V\}
\] 
for any irreducible subvariety $V \subseteq A_N$.

Now, thanks to our \cref{thm:curves_intro,thm:stronger_general_varieties} we know that for every $\alpha$ which is big enough with respect to the multiprojective degrees of $C$, the curve $f^{-1}(C)$ is transverse. This implies that the curve 
\[D := \Psi(f^{-1}(C))\] 
is also transverse, because $\Psi$ is an automorphism. Hence, we can apply to $D$ a theorem of Galateau \cite{Galateau_2010}, which provides a lower bound for the essential minimum of any transverse subvariety $V \subseteq A_N$. 

More precisely, \cite[Corollaire~1.2]{Galateau_2010} shows that
\begin{equation} \label{ineq:Galateau}
	\hat{\mu}(V) \geq \frac{c_6(A_N)}{\deg(V)^{1/\codim(V)} \log(3 \deg(V))^{\lambda(N,\codim(V))}},
\end{equation}
where $c_6(A_N) \in \mathbb{R}_{> 0}$ is an effectively computable constant depending only on $A_N$, and where one defines $\lambda(N,k) := (5 N (k+1))^{k+1}$.
To make this lower bound more explicit when $V = D$, note that
\begin{equation} \label{ineq:degree_bound}
	\deg(D) \leq 3 N^2 d_L \deg(f^{-1}(C)) \leq 3 N^3 d_L \alpha^{2 (N-r)} \deg(C) \leq 3 N^3 \alpha^{2 (N+1-r)} \deg(C)
\end{equation}
where the first upper bound can be proven using Bézout's classical theorem, as explained in the proof of \cite[Lemma~13.2]{Viada_2008}, while the second upper bound follows form \cref{lem:degree_preimage}, because we set two entries of the diagonal endomorphism $f$ to be equal to one.
Combining \eqref{eq:essential_minimum_image} with \eqref{ineq:Galateau}, with $V = D$, and \eqref{ineq:degree_bound}, we see that 
\begin{equation} \label{ineq:smart_essential_minimum}
	\hat{\mu}(\Phi(C)) = \alpha^2 \hat{\mu}(D) \geq  \frac{c_7(A_N,\deg(C))}{\log(d_L \lvert \alpha \rvert)^{\lambda(N,N-1)}} (\alpha^{2 (r-1)}/d_L)^{\frac{1}{N-1}} \geq \frac{c_7(A_N,\deg(C))}{\log(d_L \lvert \alpha \rvert)^{\lambda(N,N-1)}} \alpha^{\frac{2(r-2)}{N-1}}
\end{equation}
for some constant $c_7(A_N,\deg(C)) \in \mathbb{R}_{> 0}$ which depends on $A_N$ and on the degree of $C$. In particular, if we let $\Phi$ vary by letting $\lvert \alpha \rvert \to +\infty$, while leaving $C$ fixed, we see that the essential minimum of the images $\Phi(C)$ will tend to infinity as a power of $\lvert \alpha \rvert$, unless $r=2$. In this case, if $d_L \gg \alpha^2/\log(\lvert \alpha \rvert)$ then the lower bound for $\hat{\mu}(\Phi(C))$ portrayed in \eqref{ineq:smart_essential_minimum} will decrease as a power of $1/\log(d_L)$.

The lower bound \eqref{ineq:smart_essential_minimum} that we obtained is much better than what would come out of a direct application of Galateau's inequality \eqref{ineq:Galateau} to the curve $\Phi(C)$. More precisely, without further knowledge on $\Phi$, we know only that \[
\deg(\Phi(C)) \leq 3 N^2 \alpha^2 \deg(C),\] 
as follows again from an application of Bézout's classical theorem, which is explained in the proof of \cite[Lemma~13.2]{Viada_2008}. Combining this upper bound with \eqref{ineq:Galateau}, where we set $V = \Phi(C)$, we get the lower bound
\[
	\hat{\mu}(\Phi(C)) \geq c_8(A_N,\deg(C)) \frac{1}{\alpha^{2/(N-1)} (\log\lvert \alpha \rvert)^{\lambda(N,N-1)}}
\]
for some constant $c_8$ depending on $A_N$ and $\deg(C)$. In particular, this lower bound tends to zero as $\lvert \alpha \rvert \to +\infty$, and is evidently seen to be worse than \eqref{ineq:smart_essential_minimum}.

Lower bounds for the essential minimum such as the ones obtained in this subsection are known to be essential in the study of the Torsion Anomalous Conjecture, as shown in \cite{Viada_2008,CVV_trans}.
In particular, let us observe that, up to torsion and to a reordering of the variables, every subgroup $B \subseteq A_N$ of codimension $r \geq 2$ is of the form $B = \ker(\eta)$, where $\eta = ( \mathrm{Id}_r \ \vdots \ L ) \colon A_N \to E_1 \times \dots \times E_r$ is a morphism of abelian varieties defined as the first $r$ rows of the endomorphism $\Phi \colon A_N \to A_N$ introduced in \eqref{eq:Phi}.
Hence, the results obtained in this subsection provide an explicit link between the irreducibility statements proven in this paper and the Torsion Anomalous Conjecture, which will be investigated further in future work.


\begin{thebibliography}{BHMZ10}

\bibitem[BMZ99]{Bombieri_Masser_Zannier_1999}
E. Bombieri, D. Masser, and U. Zannier, \textit{Intersecting a curve with algebraic subgroups of multiplicative groups}, International Mathematics Research Notices, vol. 1999, no. 20, pp. 1119–1140, 1999.

\bibitem[BMZ07]{BMZ1}
E. Bombieri, D. Masser, and U. Zannier, \textit{Anomalous Subvarieties—Structure Theorems and Applications}, International Mathematics Research Notices, vol. 2007, p. rnm057, 2007.

\bibitem[CVV17]{CVV_trans}
S. Checcoli, F. Veneziano, and E. Viada, \textit{On the explicit Torsion Anomalous Conjecture}, Trans. Amer. Math. Soc., vol. 369, no. 9, pp. 6465–6491, 2017.

\bibitem[CVV19]{CVV_Forum}
S. Checcoli, F. Veneziano, and E. Viada, \textit{The explicit Mordell Conjecture for families of curves}, Forum of Mathematics, Sigma, vol. 7, p. e31, 2019.

\bibitem[Ga10]{Galateau_2010}
A. Galateau, \textit{Une minoration du minimum essentiel sur les variétés abéliennes}, Commentarii Mathematici Helvetici, vol. 85, no. 4, pp. 775–812, 2010.

\bibitem[Hab09]{Habegger_2009}
P. Habegger, \textit{On the Bounded Height Conjecture}, International Mathematics Research Notices, vol. 2009, no. 5, pp. 860–886, Jan. 2009.

\bibitem[Hin88]{grado} M. Hindry, \textit{Autour d’une conjecture de Serge Lang}, Invent Math, vol. 94, no. 3, pp. 575–603, 1988.

\bibitem[Hus04]{Husemoller_2004}
D. Husemöller, \textit{Elliptic curves}, Second Edition. Graduate Texts in Mathematics, Vol. 111. Springer-Verlag, New York, 2004.

\bibitem[Mau08]{Maurin_2008}
G. Maurin, \textit{Courbes algébriques et équations multiplicatives}, Math. Ann., vol. 341, no. 4, pp. 789–824, 2008.

\bibitem[Phi81]{patrice_zeros}
P. Philippon, \textit{Lemmes de zéros dans les groupes algébriques commutatifs}, Bulletin de la Société Mathématique de France, vol. 114, pp. 355–383, 1986.

\bibitem[Phi91]{patriceI}
P. Philippon, \textit{Sur des hauteurs alternatives. I}, Math. Ann., vol. 289, no. 1, pp. 255–283, 1991.

\bibitem[Phi95]{patrice}
P. Philippon, \textit{Sur des hauteurs alternatives. III}, J. Math. Pures Appl. (9), vol. 74, no. 4, pp. 345–365, 1995.

\bibitem[Sut22]{Sutherland_2022}
A. V. Sutherland, \textit{Elliptic curves}. MIT Course 18.783 Lecture Notes, 2022.

\bibitem[Via08]{Viada_2008}
E. Viada, \textit{The intersection of a curve with a union of translated codimension-two subgroups in a power of an elliptic curve}, Algebra \& Number Theory, vol. 2, no. 3, pp. 249–298, 2008.   

\bibitem[Via09]{Viada_2009}
E. Viada, \textit{Nondense Subsets of Varieties in a Power of an Elliptic Curve}, International Mathematics Research Notices, vol. 2009, no. 7, pp. 1213–1246, 2009.

\bibitem[Via16]{Viada_2016}
E. Viada, \textit{Explicit height bounds and the effective Mordell-Lang conjecture}, Riv. Mat. Univ. Parma (N.S.), vol. 7, no. 1, pp. 101–131, 2016.

\bibitem[Via21]{Viada_2021}
E. Viada, \textit{A criterion for transversality of curves and an application to the rational points}, Ann. Math. Québec, vol. 45, no. 2, pp. 453–464, 2021.

\bibitem[Zan12]{libroZannier} 
U. Zannier, \textit{Some Problems of Unlikely Intersections in Arithmetic and Geometry} (AM-181). With six appendices by David Masser. Princeton University Press, 2012.

\bibitem[Zha95a]{Zha_Small}
S. Zhang, \emph{Small points and adelic metrics}, J. Algebraic Geom., vol. 4, no. 2, pp. 281–300, 1995.

\end{thebibliography}
\end{document}